\documentclass[reqno, 12pt]{amsart}

\usepackage{amssymb}
\usepackage{amsmath}
\usepackage{amsthm}
\usepackage{geometry}
\usepackage[bookmarkstype=toc, colorlinks=true, linkcolor=blue, pdfborder={0 0 0}, bookmarks=true, bookmarksnumbered=true]{hyperref}

\newtheorem{theorem}{Theorem}[section]
\newtheorem{lemma}[theorem]{Lemma}
\newtheorem{cor}[theorem]{Corollary}
\newtheorem{prop}[theorem]{Proposition}
\theoremstyle{definition}
\newtheorem{definition}[theorem]{Definition}
\newtheorem{example}[theorem]{Example}
\newtheorem{remark}[theorem]{Remark}

\newcommand{\dets}{(\det{S})}
\newcommand{\R}{\mathbb{R}}
\newcommand{\C}{\mathbb{C}}

\newcommand{\im}{\mathrm{Im}}
\newcommand{\re}{\mathrm{Re}}
\newcommand{\w}{\wedge}
\newcommand{\we}{\wedge}

\newcommand{\al}{\alpha}
\newcommand{\be}{\beta}

\newcommand{\ga}{\gamma}

\newcommand{\de}{\delta}

\newcommand{\e}{\epsilon}
\newcommand{\ep}{\epsilon}

\newcommand{\io}{\iota}

\newcommand{\s}{\sigma}

\newcommand{\om}{\omega}
\newcommand{\Om}{\Omega}

\newcommand{\fg}{\mathfrak{g}}
\newcommand{\gl}[1]{\mathrm{GL}(#1;\R)}

\newcommand{\fgl}[1]{\mathfrak{gl}(#1;\R)}

\newcommand{\sym}[1]{\mathrm{Sym}(#1;\mathbb{R})}
\newcommand{\psym}[1]{\mathrm{Sym}^+(#1;\mathbb{R})}
\newcommand{\so}[1]{\mathrm{SO}(#1)}
\newcommand{\fso}[1]{\mathfrak{so}(#1)}
\newcommand{\su}[1]{\mathrm{SU}(#1)}

\newcommand{\T}[1]{\mathrm{T}^{#1}}
\newcommand{\ft}[1]{\mathfrak{t}^{#1}}

\newcommand{\no}{\nonumber}
\newcommand{\pt}[2]{\frac{\partial #1}{\partial #2}}

\def\dets{(\det{S})}
\def\no{\nonumber}
\def\pt#1#2{\frac{\partial #1}{\partial #2}}

\def\R{\mathbb{R}}
\def\C{\mathbb{C}}

\def\im{\mathrm{Im}}
\def\re{\mathrm{Re}}
\def\w{\wedge}
\def\we{\wedge}

\def\al{\alpha}

\def\be{\beta}

\def\ga{\gamma}

\def\e{\epsilon}
\def\ep{\epsilon}

\def\s{\sigma}

\def\om{\omega}
\def\Om{\Omega}
\def\fg{\mathfrak{g}}
\def\gl#1{\mathrm{GL}(#1;\R)}

\def\fgl#1{\mathfrak{gl}(#1;\R)}

\def\sym#1{\mathrm{Sym}(#1;\mathbb{R})}
\def\psym#1{\mathrm{Sym}^+(#1;\mathbb{R})}
\def\so#1{\mathrm{SO}(#1)}
\def\fso#1{\mathfrak{so}(#1)}
\def\su#1{\mathrm{SU}(#1)}

\def\ft#1{\mathfrak{t}^{#1}}
\def\rm#1{\mathrm{#1}}

\begin{document}

\title{$G_2$-metrics arising from non-integrable special Lagrangian fibrations}

\author{Ryohei Chihara}
\address{Graduate School of Mathematical Sciences, The University of Tokyo, 3-8-1 Komaba Meguro-ku Tokyo 153-8914, Japan}
\email{rchihara@ms.u-tokyo.ac.jp}

\begin{abstract}
{We study special Lagrangian fibrations of $\su3$-manifolds, not necessarily torsion-free. In the case where the fiber is a unimodular Lie group $G$, we decompose such $\su3$-structures into triples of solder 1-forms, connection 1-forms and equivariant $3\times3$ positive-definite symmetric matrix-valued functions on principal $G$-bundles over 3-manifolds. As applications, we describe regular parts of $G_2$-manifolds that admit Lagrangian-type 3-dimensional group actions by constrained dynamical systems on the spaces of the triples in the cases of $G=\T3$ and $\so3$.}
\end{abstract}

\maketitle
\setcounter{tocdepth}{1}
\tableofcontents
\numberwithin{equation}{section}


\section{Introduction}
\label{sec:int}

The geometry of $G_2$-structures on $7$-manifolds is closely related to that of $\su3$-structures on $6$-manifolds. For a one-parameter family $(\om(t),\psi(t))$ of $\su3$-structures on a 6-manifold $X$, the $3$-form $\om(t)\w dt + \psi(t)$ is a $G_2$-structure on $X\times (t_1,t_2)$. Here $\om(t)$ and $\psi(t)$ denote  the $2$- and $3$-form on $X$ defining an $\su3$-structure for each $t\in (t_1,t_2).$
Conversely, any $G_2$-structure on $Y$ is locally described by one-parameter families of $\su3$-structures on  $6$-dimensional hypersurfaces in $Y$ as above. This viewpoints has been studied by many authors \cite{BCFG, Br1,CS,Hi}. 

In the present paper, we study torsion-free $G_2$-structures given in terms of one-parameter families of $\su3$-structures on $\T3$- or $\so3$-bundles; the fibrations are special Lagrangian in the sense that the 2- and 3-forms defining each $\su3$-structure vanish along the fibers. 

Let $G$ be a connected $3$-dimensional Lie group with Lie algebra $\fg$, and $P$ the total space of a principal $G$-bundle over a $3$-manifold $M$. Denote by $(\om, \psi)$ $G$-invariant 2- and 3-forms defining a special Lagrangian $\su3$-structure on $P$ (in the sense above). We first prove that if $G$ is unimodular then such an $\su3$-structure decomposes uniquely into a triple $(e,a,S)$ of a solder 1-form $e$, a connection 1-form $a$ and a $G$-equivariant $3\times3$ positive-definite symmetric matrix-valued function $S$ on $P$ (Theorem \ref{thm:slag1}). 

Using this decomposition, we describe locally $\T3$- and $\so3$-invariant torsion-free $G_2$-structures whose definite 3-forms vanish along the fibers as orbits of constrained dynamical systems on the space of the triples $(e,a,S)$. Here a definite 3-form is the 3-form defining a $G_2$-structure on a 7-manifold. Let $(\om(t),\psi(t))$ be a one-parameter family of such $\su3$-structures defined on $(t_1,t_2)$, and $(e(t),a(t),S(t))$ the triples corresponding to $(\om(t),\psi(t))$. Then 
the 3-form $\om(t)\w \dets^{\frac{1}{2}}dt+\psi(t)$ is a $G_2$-structure on $P \times (t_1,t_2)$. From now, we omit the symbol of summation adopting Einstein's convention and denote by $\e_{ijk}$ the Levi-Civita symbol for the permutation of $\{1,2,3\}$. Besides, abbreviate $\hat{e}^i=\dfrac{1}{2}\e_{ijk}e^j\w e^k$. For a basis $\{X_1,X_2,X_3\}$ of $\fg$, we write $e=e^iX_i$ and $a=a^iX_i$. Our main results are as follows.

In the case of $\T3$-fibrations over $M$, we prove

\begin{theorem}\label{thm:a}
Let $G=\T3$, and $\{X_1,X_2,X_3\}$ a basis of the Lie algebra $\ft3$. The $G_2$-structure $\om(t)\w (\det S)^{\frac{1}{2}}dt+\psi(t)$ is torsion-free if and only if the triple $(e(t),a(t),S(t))$ is an orbit of the following constrained dynamical system on the space of triples $(e,a,S)$:
\begin{align} 
&de^i=0, \quad \Om_{ij}=\Om_{ji},\quad S_{i\al,\al}=0 \quad\text{(constraint conditions)};\label{eq:t2}\\
&\pt{e^i}{t}=0,\quad \pt{a^i}{t}=-\e_{i\al\be}\tilde{S}_{k\al,\be}e^k,\quad\pt{S_{ij}}{t}=-\Om_{ij}\quad\text{(equations of motion)}\label{eq:t1}
\end{align}
for $i,j=1,2,3$. Here $da=\Om_{ij}\hat{e}^jX_i$, $dS_{ij}=S_{ij,k}e^k$ and $\tilde{S}=\dets\cdot S^{-1}$.
\end{theorem}

Moreover by scaling, we obtain

\begin{theorem}\label{thm:a1}
Every torsion-free $\T3$-invariant $G_2$-structure whose definite 3-form vanishes along the fibers is locally given by some orbit of the constrained dynamical system in Theorem \ref{thm:a}.
\end{theorem}

\begin{remark}In  \cite{MS}, 
Madsen and Swann  already obtained results similar to Theorem \ref{thm:a} and \ref{thm:a1}. See Remark \ref{rem:ms}.
\end{remark}

In the case of $\so3$-fibrations over $M$, we prove

\begin{theorem}\label{thm:b}
Let $G=\so3$, and $\{Y_1,Y_2,Y_3\}$ a basis of the Lie algebra $\fso3$ satisfying $[Y_i,Y_j]=\e_{ijk}Y_k$ for $i,j=1,2,3.$ The $G_2$-structure $\om(t)\w (\det S)^{\frac{1}{2}}dt+\psi(t)$ is torsion-free if and only if the triple $(e(t),a(t),S(t))$ is an orbit of the following constrained dynamical system on the space of triples $(e,a,S)$:
\begin{align}
    &d_He=0, \quad S_{i\al;\al}=0 \quad (\text{constraint conditions});\label{eq:const}\\
    &\pt{e^i}{t}=\tilde{S}_{ij}e^j,\quad \pt{S}{t}=-\mathrm{tr}(\tilde{S})S +2\dets I -\Om \quad \text{(equations of motion)}\label{eq:motion}
\end{align}
for $i=1,2,3$. Here $I=(\de_{ij})$, and $d_H$ denotes the covariant derivation for each connection $a(t)$. Also $S_{ij;k}$ and $\Om_{ij}$ are defined by  $d_HS_{ij}=S_{ij;k}e^k$ and $d_Ha=\Om_{ij}\hat{e}^jY_i$.
\end{theorem}

\begin{remark}
The condition $d_He=0$ says that the connection $a$ is the Levi-Civita one of the metric on $M$ given by the local orthonormal coframe $e=e^iY_i$. Then $\Om_{ij}e^i\otimes e^j$ coincides with the Einstein tensor. 
\end{remark}

By scaling, we obtain

\begin{theorem}\label{thm:b1}
Every torsion-free $\so3$-invariant $G_2$-structure whose definite 3-form vanishes along the fibers is locally given by some orbit of the constrained dynamical 
system in Theorem \ref{thm:b}.
\end{theorem}

\begin{remark} 
We can see that equations of motion preserve the constraint conditions in Theorem \ref{thm:a} and \ref{thm:b}. This is immediate for $G=\T3$. See (\cite{Chi}, Proposition 7) for $G=\so3$.
\end{remark}

The present paper is organized as follows. In Section \ref{sec:pre}, we review definitions and some basic results of $\su3$- and $G_2$-structures. In Section \ref{sec:slag}, we consider $G$-invariant special Lagrangian fibrations of $\su3$-manifolds and prove the decomposition (Theorem \ref{thm:slag1}). We apply this theorem to $G$-invariant $G_2$-structures whose definite 3-form vanishes along the fibers in Section \ref{sec:red}. In Section \ref{sec:loc_t3} and \ref{sec:so3}, as applications of the above results, we describe locally $\T3$- and $\so3$-invariant $G_2$-manifolds whose definite 3-forms vanish along the fibers as orbits of constrained dynamical systems on the spaces of the triples (Theorem \ref{thm:a}, \ref{thm:a1}, \ref{thm:b} and \ref{thm:b1}).

\subsection*{Conventions} 

We omit the symbol of summation adopting Einstein's convention, and often abbreviate $a\w b= ab$ and $c^i\w c^j= c^{ij}$. Also we use the Levi-Civita symbol $\e_{ijk}$ and write $\hat{c}^i=(1/2)\e_{ijk}c^{jk}$ for a triple of 1-forms $\{c^1,c^2,c^3\}$. Denote by $\tilde{A}$ the adjugate matrix of an $n\times n$ matrix $A$ satisfying $\tilde{A}A=\det{(A)}I$. Here $I=(\delta_{ij})$ is the identity matrix. Let $G$ be a connected 3-dimensional Lie group with Lie algebra $\fg$, and $P\to M$ a principal $G$-bundle over a 3-manifold $M$. An equivariant $\fg$-valued 1-form $e$ with respect to the adjoint action on $\fg$ is called a {\it solder 1-form} if $e=e^iX_i$ satisfies $e^{123} \neq 0$ at each $u\in P$ for a basis $\{X_1,X_2,X_3\}$ of $\fg$.


\section{$\su3$- and $G_2$-structures}
\label{sec:pre}

In this section we review $\su3$-structures on 6-manifolds and $G_2$-structures on 7-manifolds, emphasizing relations between two structures. Throughout this paper, we assume all objects are of class $C^{\infty}$.

\subsection{$\su3$-structures}

Let $X$ be a 6-manifold, $\mathrm{Fr}(X)$ the frame bundle over $X$, which is a principal $\gl6$-bundle over $X$. 
We have the natural inclusion $\su3 \subset \gl6$ by the standard identification $\R^6 \cong \C^3$, where $z^i=x^i+\sqrt{-1}y^i$ for $i=1,2,3.$
A subbundle of $\mathrm{Fr}(X)$ is said to be an {\it $\su3$-structure} on $X$ if the structure group is contained in $\su3$.
For an $\su3$-structure on $X$, we have the associated real 2-form $\om$ and real 3-form $\psi$ pointwisely isomorphic to $\omega_0 = \sum_{i=1}^{3}dx^i\w dy^i$ and $\psi_0 = \mathrm{Im}(dz^1dz^2dz^3)$ on $\C^3$, respectively. Here $\im(*)$ denotes the imaginary part of $*$. We can identify an $\su3$-structure with such a pair $(\om,\psi)$. 
Using this identification, for an $\su3$-structure $(\om,\psi)$, we define by $\psi^{\#}$ the real 3-form on $X$ corresponding to  $\re(dz^1\w dz^2\w dz^3)$. Here $\re(*)$ denotes the real part of $*$, and $\psi^{\#}$ is the same as $-\hat{\psi}$ in \cite{Hi}.

It is useful to compare general cases with the following basic example.

\begin{example}\label{ex:R^6}
Let $X$ be $\R^6\cong \C^3$. We have
\begin{align*}
dz^{123} = (dx^{123} - \sum_{k=1}^{3}dx^k\widehat{dy}^k) + \sqrt{-1}(-dy^{123} + \sum_{k=1}^{3}dy^k\widehat{dx}^k),
\end{align*}
where 
\begin{align*}
\widehat{dy^k} = \frac{1}{2}\sum_{i,j}^{ }\epsilon_{ijk}dy^idy^j \quad \text{and} \quad \widehat{dx^k}= \frac{1}{2}\sum_{i,j}\epsilon_{ijk}dx^idx^j
\end{align*}
for $k=1,2,3$. 
Thus the forms $\om_0,\psi_0$ and $\psi_0^{\#}$ associated with the standard $\su3$-structure on $\R^6$ are expressed as follows.
\begin{align*}
&\omega_0= \sum_{k=1}^{3}dx^kdy^k ,\\
&\psi_0 = \mathrm{Im}(dz^1dz^2dz^3)= -dy^{123} + \sum_{k=1}^{3}dy^k\widehat{dx}^k,\\
&\psi_0^{\#}=\re(dz^1dz^2dz^3)=dx^{123} - \sum_{k=1}^{3}dx^k\widehat{dy}^k.
\end{align*}
\end{example}

Many authors (including \cite{Br1,Ca,CS,Gr}) studied $\su3$-structures satisfying some integrability conditions. In particular, the following conditions are important.

\begin{definition}\label{def:tor_su3}
Let $(\om,\psi)$ be an $\su3$-structure on $X$. 
\begin{enumerate}

\item $(\omega, \psi)$ is said to be {\it  half-flat} if $d\left(\omega\wedge\omega\right)= 0$ and $d\psi = 0$.

\item $(\omega,\psi)$ is said to be {\it torsion-free} if $d\omega =0$ and $d\psi = d\psi^{\#}=0$.

\end{enumerate}
\end{definition}

\begin{remark}\label{rem:su}
In general, a $G$-structure $F$ on an $n$-manifold $N$ is said to be {\it torsion-free} if the tautological 1-form $\theta \in \Om^1(F;\R^n)$ satisfies $d_H\theta :=d\theta+a\w \theta=0$ for some connection 1-form $a \in \Om^1(F;\rm{Lie}(G))\subset \Om^1(F;\fgl{n})$. The torsion-free condition in Definition \ref{def:tor_su3} is known to be equivalent to this general definition of the torsion-free condition for $G=\su3$.
\end{remark} 

By Remark \ref{rem:su} and the Ambrose-Singer theorem, we see that  an $\su3$-structure $(\omega,\psi)$ is torsion-free if and only if the associated metric $h_{(\om,\psi)}$ has the holonomy group contained in $\su3$. Moreover, If an $\su3$-structure is torsion-free, then the associated metric on $X$ is Ricci-flat.
A Riemannian metric $h$ on $X$ is said to be {\it holonomy $\su3$} if the holonomy group coincides with $\su3$.


\subsection{$G_2$-structures}

Let $Y$ be a 7-manifold, $\mathrm{Fr}(Y)$ the frame bundle over $Y$, which is a principal $\gl7$-bundle over $Y$. Let us define $G_2$-structures on $Y$ as in the above subsection. The Lie group $G_2$ is defined as the linear automorphism group of the standard definite $3$-form $\phi_0$ (presented in Example \ref{ex:R^7}) on $\R^7$. It is known that this group coincides with the linear automorphism group of the cross product structure on $\im\mathbb{O}$, where $\im\mathbb{O}$ denotes the 7-dimensional imaginary part of the octonion algebra $\mathbb{O}$.
A subbundle of $\mathrm{Fr}(Y)$ is said to be a {\it $G_2$-structure} on $Y$ if the structure group is contained in $G_2$.
For a $G_2$-structure  on $Y$, we have the associated real 3-form $\phi$ on $Y$ pointwisely isomorphic to $\phi_0$ on $\R^7$. Such a 3-form is called a {\it definite 3-form}. Then we identify the definite 3-form $\phi$ with a $G_2$-structure on $Y$ as in the above subsection. Since $G_2 \subset \mathrm{SO}(7)$, we have the Riemannian metric $g_\phi$ and  orientation associated with a $G_2$-structure $\phi$ on $Y$. Besides, we denote by $\star_{\phi}$ the Hodge star associated with $\phi$, and simply write $\star$ in situations without confusion.

The following example is the model of $G_2$-structures as in Example \ref{ex:R^6}.

\begin{example}\label{ex:R^7}
Let $Y$ = $\R^7$. Denote the standard coordinate by $(x^0,x^1,y^1,x^2,y^2,x^3,y^3)$ and fix the orientation by $dx^0dx^1dy^1dx^2dy^2dx^3dy^3$.  The standard $3$-form is
\begin{align*}
\phi_0&={} -dy^{123}+ dy^1(dx^{01}+dx^{23})+dy^{2}(dx^{02}+dx^{31})+dy^3(dx^{03}+dx^{12}), \\
&={} \om_0 \wedge dx^0 + \psi_0.
\end{align*} 
Also we have
\begin{align*}
\star_{\phi_0}\phi_0&={}dx^{0123}-dy^{23}(dx^{01}+dx^{23})-dy^{31}(dx^{02}+dx^{31}) -dy^{12}(dx^{03}+dx^{12})\\
&={} -\psi_0^{\#}\w dx^0 + \frac{1}{2}\om_0\w\om_0.
\end{align*}
Here $(\om_0,\psi_0)$ is the $\su3$-structure on $\R^6$ in Example \ref{ex:R^6}.
\end{example}

$G_2$-structures satisfying some integrability conditions are studied many authors, including \cite{Br2,CS,Fr}.

\begin{definition} \label{def:tr_G2}Let $\phi$ be a $G_2$-structure on $Y$.
\begin{enumerate}
\item $\phi$ is said to be {\it closed} if $d\phi=0$.

\item $\phi$ is said to be {\it coclosed} if $d\star_{\phi}\phi=0$.

\item $\phi$ is said to be  {\it torsion-free} if $d\phi=0$ and $d(\star_{g_{\phi}}\phi)=0$.

\end{enumerate}
\end{definition}

A $G_2$-structure $\phi$ is torsion-free if and only if the associated Riemannian metric $g_{\phi}$ has the holonomy group contained in $G_2$. If a $G_2$-structure is torsion-free then the associated metric on $Y$ is Ricci-flat. A 7-manifold $Y$ with a torsion-free $G_2$-structure $\phi$ is called a {\it $G_2$-manifold}, and 
a Riemannian metric $g$ on $Y$ is called {\it holonomy $G_2$} if the holonomy group coincides with $G_2$. 

It is useful the following lemma for a normal form of a $G_2$-structure at a point $y\in Y$. Let $\phi$ be a $G_2$-structure on $Y$, and $\{V_1,V_2,V_3\}$ orthonormal tangent vectors at $y \in Y$. Also denote by $\{V^1, V^2,V^3\}$ the dual cotangent vectors with respect to the metric $g_{\phi}$, i.e., $V^i(*)=g_{\phi}(V_i,*)$ for $i=1,2,3$. 
Let $\phi_y$ and $\star_{\phi} \phi_y$ denote 3- and 4-forms at $y\in Y$, respectively.

\begin{lemma}\label{lem:can1}
Suppose $\phi(V_1,V_2,V_3) = 0$ at $y\in Y$. Defining cotangent vectors
\begin{align*}
Z = -(\star_{\phi}\phi)(V_1,V_2,V_3,*) \quad{and}\quad E^i = \frac{1}{2}\sum_{j,k}\epsilon_{ijk}\iota(V_k)\left(\iota(V_j)\phi\right)
\end{align*}
for $i=1,2,3$, we have 
\begin{align*}
&\phi_y = \sum_{k}V^k E^k \wedge Z -(E^{123} - \sum_{k}E^k\hat{V}^{k}), \\
&\star_{\phi} \phi_y = -(V^{123}- \sum_{k}V^k\hat{E}^k)\wedge Z - \sum_{k}\hat{V}^k\hat{E}^k.
\end{align*}
\end{lemma}

\begin{proof}
We can directly check this for $V_1=\dfrac{\partial}{\partial x^1}, V_2=\dfrac{\partial}{\partial x^2}$ and $V_3=\dfrac{\partial}{\partial x^3}$ in the case of $\phi_0$ in Example \ref{ex:R^7}. This suffices to prove Lemma \ref{lem:can1} since it is known that the Lie group $G_2$ acts transitively (and faithfully) on the set of triple vectors $\{(v_1,v_2,v_3)\}$ of $\R^7$ satisfying $\phi_0(v_1,v_2,v_3)=0$ (\cite{HL}, p.~115, Proposition 1.10).
\end{proof}

\subsection{Relations between $\su3$-structures and $G_2$-structures}
$G_2$-structures are related to $\su3$-structures as stated below (See \cite{Br3,Hi} for more details). These propositions can be proved by direct calculation. Let $X$ be a 6-manifold, and consider a one-parameter family $(\om(t),\psi(t))$ of $\su3$-structures on $X$ defined on an interval $(t_1,t_2)$. 

\begin{prop}\label{prop:sug2}
The $3$-form $\phi=\om(t)\w dt + \psi(t)$ on the 7-manifold $X \times (t_1,t_2)$ is a $G_2$-structure. Also the Hodge dual is $\star_{\phi}\phi=-\psi(t)^{\#}\w dt + \frac{1}{2}\om(t)\w\om(t)$. 
\end{prop}

 Let $\phi$ be the $G_2$-structure on $X\times (t_1,t_2)$ in Proposition \ref{prop:sug2}. The torsion-free condition for $\phi$ is interpreted as the following constrained dynamical system on the space of $\su3$-structures on $X$.

\begin{prop}\label{prop:flow}
The $G_2$-structure $\phi$ is torsion-free if and only if $(\om(t),\psi(t))$ satisfies the following four equations at every $t\in(t_1,t_2)$: 
\begin{enumerate}
    \item constraint conditions
    \begin{align}d\left(\omega\wedge\omega\right)= 0 \quad\text{and}\quad d\psi = 0;\no
    \end{align}
    
    \item equations of motion 
    \begin{align}
\frac{\partial \psi}{\partial t}= d\om \quad\text{and}\quad \frac{\partial }{\partial t}\left(\frac{1}{2}\om\we\om\right)= d\psi^{\#}.\no
\end{align}
\end{enumerate}
\end{prop}

\begin{remark}
It is clear that solutions of the equations of motion preserve the half-flat conditions (constraint conditions).
\end{remark}


\section{$G$-invariant non-integrable special Lagrangian fibrations}\label{sec:slag}

Let $G$ be a connected $3$-dimensional Lie group with Lie algebra $\mathfrak{g}$, and $P$ the total space of a principal $G$-bundle $\pi:P \to M$ over a 3-manifold $M$. In this section we introduce 
our main objects, {\it $G$-invariant non-integrable special Lagrangian fibered $\su3$-structures} on $P$. Next, we present examples of such $\su3$-structures. Finally, we prove that if $G$ is unimodular then every such $\su3$-structure is uniquely constructed by a triple $(e,a,S)$ of a solder 1-form $e$, a connection 1-form $a$ and an equivariant $\psym3$-valued function $S$ on $P$. Here denote by $\sym3$ and $\psym3$ the spaces of $3\times3$ symmetric and positive-definite symmetric matrices.

\subsection{Definition}

Let us start with the definition.

\begin{definition}\label{def:3a}
An $\su3$-structure $(\omega,\psi)$ on $P$ is said to be  a {\it $G$-invariant non-integrable special Lagrangian fibered $\su3$-structure} if it satisfies:
\begin{enumerate}
\item $\omega$ and $\psi$ are invariant for the right action of $G$ on $P$;

\item the restrictions of $\omega$ and $\psi$ to the fibers $F_m \subset P$ vanish: $\omega|_{F_m}=0$ and $\psi|_{F_m}=0$ for every $m \in M$.
\end{enumerate}
\end{definition}

Hereafter, we simply refer to such $\su3$-structures as {\it $G$-invariant sLag $\su3$-structures}.

Next we present some examples. Let $V$ be a vector space with a representation $\rho: G \to \rm{GL}(V)$. Then denote by $\Om^p(P;V)^{G}$ and $\Om^p(P;V)^{G}_{hor}$ the space of $G$-equivariant $V$-valued $p$-forms and horizontal $p$-forms on $P$.

\begin{example}\label{ex:t31}
Let $G=\T3$. Fix a basis $\{X_1,X_2,X_3\}$ of $\mathfrak{t}^3$. By this basis, $\mathfrak{t}^3\cong\R^3$ and $\mathrm{Ad}(g)=I\in \gl{3}$ for every $g\in \T3$. Let $\T3$ act on $\psym3$ trivially. Given a triple $(e,a,S)$ of a solder 1-form $e=e^k X_k \in \Om^1(P;\mathfrak{t}^3)^{\T3}_{hor}$, a connection 1-form $a=a^k X_k \in \Om^1(P;\mathfrak{t}^3)^{\T3}$ and an equivariant $\psym3$-valued function $S=(S_{ij}) \in \Om^0(P;\psym3)^{\T3}$, we can see
the following 2-form $\om$ and 3-form $\psi$ 
\begin{align}
 \omega &= (\det S)^{-\frac{1}{2}}\sum_{i,j}\tilde{S}_{ij}a^ie^j, \label{eq:su3a}\\
 \psi &= -(\det{S})e^{123} + \sum_{k}e^k\hat{a}^k\label{eq:su3b}
\end{align}
is a $\T3$-invariant sLag $\su3$-structure on $P$. Here, by definition, $\Om^1(P;\mathfrak{t}^3)^{\T3}_{hor}=\Om^1(M;\R^3)$ and\\
$\Om^0(P;\psym3)^{\T3}=\Om^0(M;\psym3)$.
\end{example}

We can generalize this examples as follows.

\begin{example}\label{ex:so3}
Let $G$ be unimodular. A Lie group $G$ is called {\it unimodular} if $\det{(\mathrm{Ad}(g))}=\pm{1}$ for every $g\in G$. Fix a basis $\{X_1,X_2,X_3\}$ of $\fg$. By this basis, $\fg\cong\R^3$ and $\mathrm{Ad}(g)\in \gl3$. Let $G$ act on $\psym3$ by $g\cdot S =\mathrm{Ad}(g)\cdot S\cdot{}^t\mathrm{Ad}(g)$ for $g\in G$ and $S\in \psym3$, where $^t\mathrm{Ad}(g)$ is the transverse matrix of $\mathrm{Ad}(g)$.  Given a triple $(e,a,S)$ of a solder 1-form $e=e^kX_k \in \Om^1(P;\fg)^{G}_{hor}$, a connection 1-form $a=a^kX_k\in \Om^1(P;\fg)^G$ and an equivariant $\psym3$-valued function $S=(S_{ij}) \in \Om^0(P;\psym3)^G$, we can see that, in the same way as Example \ref{ex:t31}, the 2-form $\om$ and 3-form $\psi$ defined by (\ref{eq:su3a}) and (\ref{eq:su3b}) 
is a $G$-invariant sLag $\su3$-structure on $P$. See also the proof of Theorem \ref{thm:slag1}.
\end{example}

\subsection{Decomposition} 

Next let us start with a $G$-invariant sLag $\su3$-structure $(\om,\psi)$ on $P$ and decompose it into a triple $(e,a,S)$. Fix a basis $\{X_1,X_2,X_3\}$ of $\fg$ and denote by $A^*$ the infinitesimal vector field on $P$ for $A\in \fg$. 

First, since the Riemannian metric $h_{(\om,\psi)}$ associated with $(\om,\psi)$ is invariant for the right action of $G$ on $P$, we have a connection 1-form $a=a^k X_k \in \Om^1(P;\fg)^G$ by the orthogonal decomposition of the tangent bundle $TP$.

Second, we define 1-forms $e^k \in \Om^1(P;\R)$ by
 \begin{align*}
 e^k = \frac{1}{2}\sum_{i,j}\epsilon_{ijk}\iota(X_j^*)\left(\iota(X_i^*)\psi\right)
 \end{align*}
for $k=1,2,3.$ Finally, we obtain a $\psym3$-valued function $S$ on $P$ by the following.

  \begin{prop}\label{prop:3a} There exists uniquely a $\psym3$-valued function $S$ on $P$ 
 such that 
 \begin{align*}
 \omega &= (\det S)^{-\frac{1}{2}}\sum_{i,j}\tilde{S}_{ij}a^ie^j, \\
 \psi &= -(\det{S})e^{123} + \sum_{k}e^k\hat{a}^k.
  \end{align*}
\end{prop}

\begin{proof}
Let $U$ be an open neighborhood of any point $m \in M$ such that $\pi^{-1}(U)\cong U \times G$. Denote by $TF$ the fiberwise tangent bundle $\ker{(d\pi)}$ on $P$. Let us consider the restriction $TF|_{\pi^{-1}(U)}$ of $TF$ to $\pi^{-1}(U) \subset P$. Since $\pi^{-1}(U)$ is trivial as a principal $G$-bundle over $U$, we have $G$-invariant vector fields $\{V_1,V_2,V_3\}$ on $\pi^{-1}(U)$ that are orthonormal with respect to the metric $h_{(\om,\psi)}$ and contained in $TF|_{\pi^{-1}(U)}$. Using this vector fields and the infinitesimal vector fields $\{X_1^*,X_2^*,X_3^*\}$, we define a $\gl3$-valued function $Q=(Q_{ij})$ on $\pi^{-1}(U)$ by 
\begin{align*}
(V_k)_u = \sum_{j}Q_{jk}(X_j^*)_u
\end{align*}
for $k=1,2,3$, and at each point $u \in \pi^{-1}(U)$. Further, put $Q^{-1}=(Q^{ij})$ and define the dual 1-forms $\{V^1,V^2,V^3\}$ by 
$V^i(*)=h(V_i,*)$. In this notation, we have 
\begin{align}
V^k = \sum_{j}Q^{kj}a^j\label{eq:3d}
\end{align}
for $k=1,2,3$.

In what follows, we use the following identities for the cofactor matrix $\tilde{Q}=(\tilde{Q}_{ij})$ of $Q$. These identities hold for general regular $3\times 3$ matrices.
\begin{eqnarray*}
\tilde{Q}_{ij} &=& \frac{1}{2}\sum_{k,l,\alpha,\beta}\epsilon_{i\alpha\beta}\epsilon_{jkl}Q_{k\alpha}Q_{l\beta} ,\\
\tilde{Q}& =& \det{Q}\cdot Q^{-1}, \quad \det{\tilde{Q}}= \det{Q^2}.
\end{eqnarray*}

In order to apply Lemma \ref{lem:can1} for normal forms, let us consider the 3-form $\phi=\om\w dt +\psi$ on $P\times \R$. The 3-form $\phi$ is a $G_2$-structure by Proposition \ref{prop:sug2}. Since the orthonormal vector fields $\{V_1,V_2,V_3\}$ satisfy $\phi(V_1,V_2,V_3)=0$, we can apply Lemma \ref{lem:can1} to $\phi$. Using the same notation $E^k$ and $Z$ in Lemma \ref{lem:can1} and the above identities for the cofactor matrix $\tilde{Q}$, we have 
\begin{align}
E^k &= \frac{1}{2}\sum_{i,j}\epsilon_{ijk}\iota(V_j)\left(\iota(V_i)\phi\right) 
= \frac{1}{2}\sum_{i,j}\epsilon_{ijk}\iota(V_j)\left(\iota(V_i)\psi\right) \label{eq:3e}\\
&= \sum_{j}\tilde{Q}_{kj}e^j, \no\\
\hat{E}^k &= \det{Q} \cdot \sum_{j}Q_{jk}\hat{e}^j ,\quad \hat{V}^k = \sum_{j}\tilde{Q}^{jk}\hat{e}^j\label{eq:3f}
\end{align}
for $k=1,2,3$, where $\tilde{Q}^{-1} = (\tilde{Q}^{ij})$. Moreover, by (\ref{eq:3d}), (\ref{eq:3e}) and (\ref{eq:3f}), we have 
\begin{align*}
&\sum_{k}V^kE^k = \det{Q}\cdot({}^tQ^{-1}Q^{-1})_{ij}a^ie^j ,\\
&-E^{123} + \sum_{k}V^k\hat{E}^k= -\det{Q}^2e^{123} + \sum_{k}e^k\hat{a}^k ,\\
&V^{123} - \sum_{k}V^k\hat{E}^k = \det{Q^{-1}}a^{123} - \det{Q}\sum_{k}a^k\hat{e}^k ,\\
&Z = dt .
\end{align*}
Define $S$ by $Q \cdot^{t}Q$, where $^tQ$ is the transverse matrix of $Q$. Then, by these identities, we have consequently
 \begin{align*}
 \omega &= (\det S)^{-\frac{1}{2}}\sum_{i,j}\tilde{S}_{ij}a^ie^j, \\
 \psi &= -(\det{S})e^{123} + \sum_{k}e^k\hat{a}^k.
  \end{align*} Thus the function $S$ is our goal. Further by the definition of $S$, we can see that $S$ is independent of the choice of the orthonormal vector fields $\{V_1,V_2,V_3\}$. 
\end{proof}

By Example \ref{ex:so3} and Proposition \ref{prop:3a}, we obtain

\begin{theorem}\label{thm:slag1}
Let $G$ be unimodular. Then there exists a one-to-one correspondence between $G$-invariant sLag $\su3$-structures and triples $(e,a,S)$ in Example \ref{ex:so3}.
\end{theorem}

\begin{proof}
Let $(\om,\psi)$ be a $G$-invariant sLag $\su3$-structure on $P$ and $(e,a,S)$ the triple corresponding to $(\om,\psi)$ in Proposition \ref{prop:3a}. Then by the constructions, we have the following:
\begin{align*}
    &R_g^*a^i=g^{ij}a^j,\quad R_g^*e^i=|\mathrm{Ad}(g)|g^{ij}e^j, \\
    &R_g^*\hat{a}^i=|\mathrm{Ad}(g)|^{-1}g_{ji}\hat{a}^j, \quad R_g^*\hat{e}^i=|\mathrm{Ad}(g)|g_{ji}\hat{e}^j,\\
    &R_g^*S=\mathrm{Ad}(g)^{-1}S(^t{\mathrm{Ad}(g)}^{-1}), \quad R_g^*(\det{S})=|\mathrm{Ad}(g)|^{-2}\det{S}, \\ &R_g^*\tilde{S}=|\mathrm{Ad}(g)|^{-2}({}^t\mathrm{Ad}(g))\tilde{S}\mathrm{Ad}(g)
\end{align*}
for $i=1,2,3$ and for every $g\in G$, where $R_g:P\to P$ denotes the right action of $g$ and $\mathrm{Ad}(g)=(g_{ij})$ is the representation of the adjoint action on $\fg$ by the fixed basis $\{X_1,X_2,X_3\}$. Further $|\mathrm{Ad}(g)|$ denotes the determinant. By this, if $G$ is unimodular, then the triple $(e,a,S)$ is a triple satisfying conditions in Example \ref{ex:so3}. Moreover the converse is straightforward by the above identities. 
\end{proof}

In summary, in the cases where $G$ is unimodular, a $G$-invariant sLag $\su3$-structure $(\om,\psi)$ on $P$ uniquely corresponds to a triple $(\{e^k\},\{a^k\},S)$ of a solder 1-form $e$, a connection 1-form $a$ and an equivariant $\psym3$-valued function $S$ on $P$ such that 
\begin{align}
 \omega &= (\det S)^{-\frac{1}{2}}\sum_{i,j}\tilde{S}_{ij}a^ie^j, \label{eq:su3c}\\
 \psi &= -(\det{S})e^{123} + \sum_{k}e^k\hat{a}^k.\label{eq:su3d}
  \end{align}
Moreover, by the proof of Proposition \ref{prop:3a}, we have 
\begin{align}
&\psi^{\#} = (\det{S})^{-\frac{1}{2}}a^{123} - (\det{S})^{\frac{1}{2}}\sum_{k}a^k\hat{e}^k \label{eq:su3e}, \\
&\frac{1}{2}\omega\wedge \omega = - \sum_{i,j}S_{ij}\hat{a}^i\hat{e}^j.\label{eq:su3f}
\end{align}
Then the following 1-forms 
\begin{align*}
   \sum_j Q^{ij}a^j \quad{and}\quad \det(Q)\sum_j Q^{ij}e^j \quad{for}\quad i=1,2,3
\end{align*}
are an orthonormal coframe with respect to the metric $h_{(\om,\psi)}$ associated with $(\om,\psi)$ on $P$. Here $S=Q\cdot {}^tQ$ and $Q^{-1}=(Q^{ij})$ as above.


\section{Reduction of $G$-invariant $G_2$-manifolds}\label{sec:red}

In this section, we prove a generalization of results in Section \ref{sec:slag} to a class of $G$-invariant $G_2$-structures. 
Let $G$ be a 3-dimensional Lie group with Lie algebra $\fg$, and $Q \to N$ a principal $G$-bundle over a 4-manifold $N$. Let us define $G_2$-structures that we consider. 

\begin{definition}\label{def:lagg2}
A $G_2$-structure $\phi$ on $Q$ is said to be a {\it $G$-invariant Lagrangian-type fibered $G_2$-structure} if it satisfies
\begin{enumerate}
    \item $\phi$ is invariant for the right action of $G$ on $Q$,
    \item the restriction of $\phi$ to the fiber $F_n$ vanishes: $\phi|_{F_n}=0$ for all $n\in N$.
    \end{enumerate}
\end{definition}

In what follows, we refer to such $G_2$-structures as {\it $G$-invariant Lag $G_2$-structures}. The following example is typical.

\begin{example}
Let $P\to M$ be a principal $G$-bundle over a 3-manifold $M$, and $(\om(t), \psi(t))$ a one-parameter family of $G$-invariant sLag $\su3$-structures on $P$ defined on an interval $(t_1,t_2)$. Generally, if $f(u,t)$ is any $G$-invariant positive function on $P\times (t_1,t_2)$, then the 3-form $\om(t)\w dt+\psi(t)$ defines a $G$-invariant Lag $G_2$-structure.
\end{example}

Let $\phi$ be a $G$-invariant Lag $G_2$-structure on $Q\to N$, and fix a basis $\{X_1,X_2,X_2\}$ of $\fg$. The following is a generalization of Theorem \ref{thm:slag1} to such $G_2$-structures. 

\begin{prop}\label{prop:red_1}
Suppose $G$ is unimodular and the 1-form $\io(X_3^*)(\io(X_2^*)(\io(X_1^*)\star\phi))$ is closed. Then, for each $n \in N$, there exists a triple of a 3-dimensional submanifold $(n \in) D$ of $N$, a one-parameter family of $G$-invariant sLag $\su3$-structures on $\pi^{-1}(D)$ and a $G$-invariant function $f(u,t)$ on $\pi^{-1}(D)\times (t_1,t_2)$ such that $\phi$ is isomorphic to $\om(t)\w fdt+\psi(t)$ on some neighborhood of $\pi^{-1}(D)$ in $Q$, where $(t_1,t_2)$ denotes an interval on which the one-parameter family is defined.
\end{prop}

\begin{proof}
By assumption, there exist a neighborhood $U \subset N$ of $n$ and $G$-invariant real-valued function $\mu$ on $\pi^{-1}(U)$ such that $\mu(q)=0$ for any $q\in \pi^{-1}(n)$ and $d\mu=-\iota(X_3^*)(\iota(X_2^*)(\iota(X_1^*)\star\phi))$ on $\pi^{-1}(U)$. Take a submanifold $(n\in) D$ of N such that $\pi^{-1}(D)\subset \mu^{-1}(0)$. Denote by $v$ the vector field $\mathrm{grad}(\mu)/|\mathrm{grad}(\mu)|^{2}$ on $\pi^{-1}(U)$. Using the integral curve of $v$, define the local diffeomorphism $\Phi:\pi^{-1}(D)\times(-\ep,\ep)\to\pi^{-1}(U)$, where $\ep$ is some positive number. Then putting 
 \begin{align*}
 &\om(t)=\Phi(*,t)^{*}\iota(|\mathrm{grad}(\mu)|v)(\phi), \\
 &\psi(t)=\Phi(*,t)^{*}(\phi), \\
 &f(u,t)=\frac{1}{|\mathrm{grad}(\mu)|}
 \end{align*}
 for each fixed $t\in(-\ep,\ep)$, we have $\Phi^*\phi=\om(t)\we fdt+ \psi(t)$ on $\pi^{-1}(D)\times(-\ep,\ep)$. 
\end{proof}

\begin{remark}
Thus any $G$-invariant Lag $G_2$-structure satisfying the assumption in Proposition \ref{prop:red_1} is locally isomorphic to $\om(t)\we fdt +\psi(t)$ for some one-parameter family of $G$-invariant sLag $\su3$-structures. More strongly, we can prove that every coclosed $G$-invariant Lag $G_2$-structure is locally isomorphic to $\om(t)\we (\det{S(t)})^{\frac{1}{2}}dt +\psi(t)$, where $(e(t),a(t),S(t))$ is the one-parameter family of the triples corresponding to some $(\om(t),\psi(t))$ in cases of $G=\T3$ and $\so3$. See the proof of Proposition  \ref{prop:loc_t3_1} and \ref{prop:6_3}.
\end{remark}

\begin{remark}
As pointed out in \cite{Go}, the following properties hold. Let $G$ be unimodular and fix a basis $\{X_1,X_2,X_3\}$ of $\fg$. In general, let $Y$ be a smooth 7-manifold with an invariant $G_2$-structure $\phi$ admitting a smooth action of $G$. 
\begin{enumerate}
    \item If $\phi$ is closed, then the function $\phi(X_1^*,X_2^*,X_3^*)$ is constant.
    
    \item If $\phi$ is coclosed, then the 1-form $\io(X_3^*)(\io(X_2^*)(\io(X_1^*)\star\phi))$ is closed.
    
\end{enumerate}
The former follows from the following equation
\begin{align*}
d(\phi(X_1^*,X_2^*,X_3^*)) &={} \phi([X_1^*,X_2^*],X_3^*,*)+\phi([X_2^*,X_3^*],X_1^*,*)+\phi([X_3^*,X_1^*],X_2^*,*) \\
 &{}\quad -(d\phi)(X_1^*,X_2^*,X_3^*,*).
\end{align*}
The latter is also proved in the same way. Thus if the invariant $G_2$-structure $\phi$ is torsion-free and the action has both of irregular and regular parts, then $\phi$ is a $G$-invariant Lag $G_2$-structure on regular parts of $Y$.
\end{remark}

\begin{remark}
Let $Y$ be a 7-manifold with a $G_2$-structure $\phi$. 3-dimensional submanifolds $L$ in $Y$ satisfying the condition $\phi|_{L}=0$ in Definition \ref{def:lagg2} are studied in \cite{GS}. The authors characterized the infinitesimal deformation space of the Lagrangian-type submanifolds in $Y$. These submanifolds are called {\it maximally $\star\phi$-like submanifolds} in (\cite{HL}, II.6.).
\end{remark}


\section{Local structure of $\T3$-invariant $G_2$-manifolds}\label{sec:loc_t3}

\subsection{Torsion of $\su3$-structures in the case of $G=\T3$}\label{subsec:tor_t3}

In this subsection, we first calculate the torsion of $\T3$-invariant sLag $\su3$-structures on the total space of a principal $\T3$-bundle $\pi:P \to M$ over a $3$-manifold $M$.  Next we prove some corollaries that follow from the derived expressions. 

Let $G=\T3$, and $P$ the total space of a principal $\T3$-bundle over a 3-manifold $M$. Fix a basis $\{X_1,X_2,X_3\}$ of $\ft3$. Then as we have seen in Section \ref{sec:slag}, $\T3$-invariant sLag $\su3$-structures $(\om,\psi)$ on $P$ uniquely correspond to triples $(e,a,S)$ of solder 1-forms $e=e^kX_k \in \Om^1(M;\ft3)$, connection 1-forms $a=a^kX_k\in\Om^1(P;\ft3)^{\T3}$ and $\psym3$-valued functions $S=(S_{ij}) \in \Om^0(M;\psym3)$.
This correspondence satisfies (\ref{eq:su3c}) and (\ref{eq:su3d}).

Let $(\om,\psi)$ be a $\T3$-invariant sLag $\su3$-structure on $P$, which has the triple $(e,a,S)$ corresponding to $(\om,\psi)$. Denote by $T=\{T_{ij}\}$ and $\Om=\{\Om_{ij}\}$ the torsion and curvature forms of $(e,a)$: $de^i=T_{ij}\hat{e}^j$ and $da^i=\Om_{ij}\hat{e}^j$ for $i=1,2,3.$ Also for a matrix-valued function $A\in \Om^0(M;\mathrm{M}(k;\R))$, denote the derivatives by $dA_{ij}=A_{ij,k}e^k$.
Then, by a simple calculation, we have the following:
\begin{align}\label{eq:3_1}
d\hat{e}^i = \epsilon_{i\al\be}T_{\al\be}e^{123}, \quad d\hat{a}^i = \epsilon_{i\al\be}\Omega_{\al\gamma}\hat{e}^{\gamma}a^{\be}
\end{align}
for $i=1,2,3$.
By (\ref{eq:su3c})--(\ref{eq:su3f}) and (\ref{eq:3_1}), we obtain the following expressions for the torsion of $(\om,\psi)$ in terms of $(e,a,S)$. The proof is straightforward calculation of differential forms.

\begin{prop}\label{prop:4b}
Let $(e,a,S)$ be the triple corresponding to $(\om,\psi)$. Then we have
\begin{align*}
&d\omega = \left(\epsilon_{j\alpha\beta}((\det{S})^{-\frac{1}{2}}\tilde{S})_{i\alpha,\beta} - (\det{S})^{-\frac{1}{2}}\tilde{S}_{i\alpha}T_{\alpha j}\right)a^i\hat{e}^j + (\det{S})^{-\frac{1}{2}}\tilde{S}_{ij}\Omega_{ij}e^{123} ,\\
&d\left(\frac{1}{2}\omega \wedge \omega\right) = -\left(S_{i\alpha,\alpha}+\epsilon_{\alpha\beta\gamma}S_{i\alpha}T_{\beta\gamma}\right)\hat{a}^ie^{123} ,\\
&d\psi = T_{ij}\hat{a}^i\hat{e}^j - \epsilon_{i\alpha\beta}\Omega_{\alpha\beta}a^ie^{123} ,\\
&d\psi^{\#} =  d\left((\det{S})^{-\frac{1}{2}}\right)a^{123} + \det{S}^{-\frac{1}{2}}\Omega_{ij}\hat{a}^i\hat{e}^j   \\ &{}\quad\quad\quad - d\left((\det{S})^{\frac{1}{2}}\right)a^k\hat{e}^k + \det{S}^{\frac{1}{2}}\epsilon_{i\alpha\beta}T_{\alpha\beta}a^ie^{123}.\no
\end{align*}
\end{prop}

We have the following corollaries to Proposition \ref{prop:4b}. 

\begin{cor} \label{cor:4a}
Let $(e,a,S)$ be the triple corresponding to $(\om,\psi)$. 
\begin{enumerate}

\item The $\su3$-structure $(\omega,\psi)$ is half-flat if and only if 
\begin{eqnarray*}
&&T_{ij} =0, \quad\Omega_{ij} = \Omega_{ji}, \quad S_{ik,k} =0 \quad\text{for}\quad i,j=1,2,3.
\end{eqnarray*}

\item Suppose $(\omega,\psi)$ is half-flat. Then $d\omega =0$ if and only if 
\begin{eqnarray*}
&\epsilon_{i\alpha\beta}((\det{S})^{-\frac{1}{2}}\tilde{S})_{j\alpha,\beta}=0 ,\quad 
\tilde{S}_{\al\be}\Omega_{\al\be} = 0 \quad\text{for}\quad i,j=1,2,3.
\end{eqnarray*}

\item Suppose $(\omega,\psi)$ is half-flat. Then $d\psi^{\#} = 0$ if and only if 
\begin{eqnarray*}
d\left(\det{S}\right)=0, \quad 
\Omega_{ij} = 0\quad \text{for}\quad i,j=1,2,3 .
\end{eqnarray*}

\item The $\su3$-structure $(\omega,\psi)$ is torsion-free if and only if 
\begin{eqnarray*}
&&T_{ij} = \Omega_{ij} = 0 , \quad 
\epsilon_{i\alpha\beta}\tilde{S}_{j\alpha.\beta} = 0,\quad 
d\left(\det{S}\right)= 0\quad\text{for}\quad i,j=1,2,3.
\end{eqnarray*}

\end{enumerate}
\end{cor} 

The following corollary is consistent with the well-known fact that a $\T3$-invariant Calabi-Yau $3$-fold whose orbits are special Lagrangian submanifolds is locally constructed by a solution of the real Monge-Amp\'ere equation on a domain of $\R^3$. 

\begin{cor}\label{cor:4_2}
Let $(\omega,\psi)$ be torsion-free, and $(e,a,S)$ the triple corresponding to $(\om,\psi)$. Then, for each $m \in M$,  there exists a coordinate neighborhood $(U;x^1,x^2,x^3)$ of $M$, a trivialization $F: \pi^{-1}(U) \to U \times T^{3}$ and a function $\rho \in C^{\infty}(U;\mathbb{R})$ such that
\begin{align}
&a|_{\pi^{-1}(U)} = F^{*}(g^{-1}dg), \quad e^i|_{U} = dx^i \label{eq:4g},\\
&\tilde{S}_{ij}|_{U} = \frac{\partial^2\rho}{\partial x^i\partial x^j} \label{eq:4h},\quad \det{\left(\frac{\partial^2\rho}{\partial x^i\partial x^j}\right)} = \mathrm{const.} 
\end{align}
 for $i,j=1,2,3$, where $g^{-1}dg$ is the natural left-invariant $1$-form on $\T3$.
\end{cor}

This corollary is proved by combining the fourth in Corollary \ref{cor:4a} and Poincar\'e's lemma.

\begin{lemma}\label{lemma:4b}
Let $U$ be a contractible domain of $\R^3$ and $\left( h_{ij}\right) \in \Om^0(U;\sym3)$. Suppose that $\ep_{i\al\be}\dfrac{\partial h_{j\al}}{\partial x^{\be}}=0$ for $i,j=1,2,3$. Then there exists a function $f\in \Om^0(U;\R)$ such that $h_{ij}=\dfrac{\partial^2 f}{\partial x^i \partial x^j}$ for $i,j=1,2,3$.
\end{lemma}

\begin{proof}[Proof of Lemma \ref{lemma:4b}]
Let us take 1-forms $f_i=h_{ij}dx^j$ for $i=1,2,3$. Then, by assumption, the 1-forms are closed. Thus, by Poincar\'e's lemma, there exist functions $\varphi_i \in \Om^0(U;\R)$ such that $f_i=\dfrac{\partial \varphi_i}{\partial x^j}dx^j$ for $i=1,2,3$. Since the matrix $(h_{ij})=(\dfrac{\partial \varphi_i}{\partial x^j})$ is symmetric, we have $d(\varphi_idx^i)=0$. Hence, using Poincar\'e's lemma again, we obtain a function $f \in \Om^0(U;\R)$ such that $h_{ij}=\dfrac{\partial^2 f}{\partial x^i \partial x^j}$ for $i,j=1,2,3$. The function $f$ is our goal. 
\end{proof}

\begin{proof}[Proof of Corollary \ref{cor:4_2}]The existence of $(e,a,F)$ satisfying (\ref{eq:4g}) is deduced form the fourth in Corollary \ref{cor:4a}. A function $\rho$ satisfying (\ref{eq:4h}) is constructed by combining the fourth in Corollary \ref{cor:4a} and Lemma \ref{lemma:4b}.
\end{proof}

In \cite{Ba}, Baraglia established a generalization of the Monge-Amp\'ere equation  to $\T4$-invariant $G_2$-manifolds whose orbits are coassociative submanifolds. In the sequel, we study generalizations of Corollary \ref{cor:4_2} to $\T3$- and $\so3$-invariant Lagrangian-type fibered $G_2$-manifolds.

\subsection{Local structure of $\T3$-invariant $G_2$-manifolds}\label{subsec:loc_t3}

In this subsection, we  prove that all torsion-free $\T3$-invariant Lag $G_2$-structures are locally described by orbits of constrained dynamical systems on the spaces of the triples $(e,a,S)$.



In Section \ref{sec:red}, we proved that every coclosed $\T3$-invariant Lag $G_2$-structure was locally isomorphic to  $\om(t)\w f(t)dt + \psi(t)$ on $P \times (t_1,t_2)$, where $(\om(t),\psi(t),f(t))$ was some one-parameter family of $\T3$-invariant sLag $\su3$-structures $(\om,\psi)$ and $\T3$-invariant positive functions $f$ on some $P$, defined on some interval $(t_1,t_2)$.

Let $(e(t),a(t),S(t),f(t))$ be a one-parameter family defined on an interval $(t_1,t_2)$ of the triples corresponding to $(\om,\psi)$ and $\T3$-invariant positive functions $f$ on $P$. we use the notation in Subsection \ref{subsec:tor_t3}.

\begin{prop}\label{prop:loc_t3_1}
 The $G_2$-structure $\om(t)\w f(t)dt + \psi(t)$ is torsion-free if and only if the quadruplet $(e(t),a(t),S(t),f(t))$ satisfies the following equations:
 \begin{align*}
     &de^i=0,\quad\Om_{ij}=\Om_{ji},\quad S_{i\al,\al}=0,\quad d(f(\det S)^{-\frac{1}{2}})=0;\\
      &\pt{e^i}{t}=0,\quad \pt{a^i}{t}=-f(\det{S})^{-\frac{1}{2}}\ep_{i\al\be}\tilde{S}_{k\al,\be}e^k,\quad \frac{\partial S_{ij}}{\partial t}= -f\det{S}^{-\frac{1}{2}}\Om_{ij}
 \end{align*}
 for $i,j=1,2,3$ and for every $t\in(t_1,t_2)$.
\end{prop}

In the proof of Proposition \ref{prop:loc_t3_1}, we use the following lemmas. Let $A$ be an $n\times n$ matrix-valued function defined on an interval and $\tilde{A}$ the cofactor matrix of $A$, which is $(\det{A})A^{-1}$ when $A$ is regular. 
The proof of Lemma \ref{lem:5a} is straightforward.

\begin{lemma}\label{lem:5a} 
If we have $\dfrac{\partial A}{\partial t}=B$, then  $\dfrac{\partial (\det{A})}{\partial t}=\mathrm{tr}({\tilde{A}B})$. In particular, if we have $\dfrac{\partial S}{\partial t}=-f(\det{S})^{-\frac{1}{2}}\Om$, then $\dfrac{\partial (\det S)}{\partial t}=-f\dets^{-\frac{1}{2}}\mathrm{tr}(\tilde{S}\Om).$
\end{lemma}
 
Let $S=(S_{ij}) \in \Om^0(M;\psym3)$. Define $C=(C_{ij})$
 and $D=(D_{ij})$ by $C_{ij}=\e_{i\al\be}(\dets^{-\frac{1}{2}}\tilde{S}_{j\al})_{,\be}$ and $D_{ij}=\e_{\al ij}S_{\al\be,\be}$.
 
 \begin{lemma}\label{lem:5b}We have
 \begin{align*} 
     CS-SC=\dets^{-\frac{1}{2}}SDS.
 \end{align*}
 In particular, $CS-SC=0$ holds if $S_{i\al,\al}=0$ for $i=1,2,3$.
 \end{lemma}
 
 \begin{proof}
 Take a basis $\{Y_1,Y_2,Y_3\}$ of $\fso3$ satisfying $[Y_i,Y_j]=\e_{kij}Y_k$ for $i,j=1,2,3$ and set $b=\dets^{-\frac{1}{2}}\tilde{S}_{ij}e^jY_i$. Without loss of generality, we can assume $de=0$. Then we have 
 \begin{align}
     [b\w b] &={} [\dets^{-\frac{1}{2}}\tilde{S}_{\al\be}e^{\be}Y_{\al}\w \dets^{-\frac{1}{2}}\tilde{S}_{\ga\de}e^{\de}Y_{\ga}] \no\\
     &={} \dets^{-1}\e_{j\be\de}\e_{i\al\ga}\tilde{S}_{\al\be}\tilde{S}_{\ga\de}\hat{e}^jY_i \no\\
     &={} 2S_{ji}\hat{e}^jY_i\no.
     \end{align}
Here we use $\tilde{Q}_{ij} = \frac{1}{2}\sum_{k,l,\alpha,\beta}\epsilon_{i\alpha\beta}\epsilon_{jkl}Q_{k\alpha}Q_{l\beta}$ for any $3\times3$ matrix $Q$. Thus 
\begin{align} \label{eq:5_24}
\frac{1}{2}d[b\w b]=S_{i\al,\al}e^{123}Y_i.
\end{align}
On the other hand, we have 
\begin{align}
    \frac{1}{2}d[b\w b]&={} [db\w b]  \label{eq:5_25} \\
    &={} [(\dets^{-\frac{1}{2}}\tilde{S}_{\al\be})_{,k}e^{k\be}Y_{\al} \w\dets^{-\frac{1}{2}}\tilde{S}_{\ga\de}e^{\de}Y_{\ga} ] \no \\
    &={} [\e_{jk\be}(\dets^{-\frac{1}{2}}\tilde{S}_{\al\be})_{,k}\hat{e}^jY_{\al} \w \dets^{-\frac{1}{2}}\tilde{S}_{\ga\de}e^{\de}Y_{\ga}] \no\\
    &={} \e_{i\al\ga}\e_{jk\be}\dets^{-\frac{1}{2}}\tilde{S}_{\ga j}(\dets^{-\frac{1}{2}}\tilde{S}_{\al\be})_{,k}e^{123}Y_i \no \\
    &={} \e_{i\al\be}\dets^{-\frac{1}{2}}\tilde{S}_{\al\ga}C_{\ga\be}e^{123}Y_i.\no
\end{align}
By (\ref{eq:5_24}) and (\ref{eq:5_25}), we get $\e_{i\al\be}\dets^{-\frac{1}{2}}\tilde{S}_{\al\ga}C_{\ga\be}=S_{i\al,\al}$ for $i=1,2,3$, implying
$
    \dets^{-\frac{1}{2}}(\tilde{S}C-C\tilde{S})=D.$
Hence $\dets^{\frac{1}{2}}(CS-SC)=SDS$.
 \end{proof}

\begin{proof}[Proof of Proposition \ref{prop:loc_t3_1}]
Denote by $\dot{\s}$ the derivative $\dfrac{\partial \s}{\partial t}$ of a differential form $\s$ and set $\dot{e}^i=p_{ij}\hat{e}^j$ and $\dot{a}^i=q_{ij}\hat{e}^j$ for $i=1,2,3.$ 

First the $G_2$-structure $\om(t)f(t)dt+\psi(t)$ is torsion-free if and only if the one-parameter family $(\om(t),\psi(t),f(t))$ satisfies the half-flat condition, $\dfrac{\partial \psi}{\partial t}=d(f\om)$ and $\dfrac{\partial}{\partial t}\left(\dfrac{1}{2}\om\w\om\right)=d(f\psi^{\#})$ for each $t$ as in Proposition \ref{prop:flow}.

Let us rewrite the conditions in terms of the one-parameter family $(e(t),a(t),S(t),f(t))$ corresponding to $(\om(t),\psi(t),f(t))$. By Corollary \ref{cor:4a}, the half-flat condition is equivalent to the following:
\begin{align}\label{eq:5_26}
    de^i=0,\quad \Om_{ij}=\Om_{ji} \quad{and}\quad S_{i\al,\al}=0
\end{align}
for $i,j=1,2,3.$ By (\ref{eq:su3d}), we have 
\begin{align}
\label{eq:5_27}\frac{\partial \psi}{\partial t}&={}\frac{\partial }{\partial t}\left(-(\det{S})e^{123}+e^i\hat{a}^i\right)                \\
&={} -\frac{\partial}{\partial t}(\det S)e^{123}-(\det S)\dot{e}^i\hat{e}^i+\ep_{k\al i}\ep_{k\be j}q_{\al\be}a^i\hat{e}^j + \hat{a}^i\dot{e}^i \no\\
&={} -(\pt{\dets}{t}+\dets p_{\al\al})e^{123}+(\de_{ij}q_{\al\al}-q_{ji})a^i\hat{e}^j+p_{ij}\hat{a}^ie^j. \no
\end{align}
Here we use 
\begin{align*}
\frac{\partial }{\partial t}(e^{123})= \dot{e}^k\hat{e}^k,\quad \frac{\partial \hat{a}^i}{\partial t}=\ep_{i\al\be}\dot{a}^{\al}a^{\be}\quad{and}\quad \ep_{k\al i}\ep_{k\be j}q_{\al\be}=\de_{ij}q_{\al\al}-q_{ji}
\end{align*}
for $i,j=1,2,3$. By (\ref{eq:su3c}), we have 
\begin{align}
    \frac{\partial}{\partial t}\left(\frac{1}{2}\om \we \om\right)&={}
\frac{\partial}{\partial t}\left(-S_{ij}\hat{a}^i\hat{e}^j\right) \label{eq:5_28}\\
&={}-\ep_{i\al\be}q_{\al\ga}S_{\ga\be}a^ie^{123} - \frac{\partial S_{ij}}{\partial t}\hat{a}^i\hat{e}^j - S_{ij}\hat{a}^i\ep_{j\al\be}\dot{e}^{\al}e^{\be} \no\\
&={} -\e_{i\al\be}q_{\al\ga}S_{\al\be}a^ie^{123}+(S_{i\al}p_{j\al}-S_{ij}p_{\al\al}-\pt{S_{ij}}{t})\hat{a}^i
\hat{e}^j   \no.
\end{align}
Here we use $ -\e_{\ga\al\be}\e_{j\de\be}S_{i\ga}p_{\al\de}=S_{i\al}p_{j\al}-S_{ij}p_{\al\al}$
for $i,j=1.2.3.$ Moreover, by Proposition \ref{prop:4b} we have 
\begin{align}
 \label{eq:5_29}   d(f\om)&={} d(f\dets^{-\frac{1}{2}}\tilde{S}_{ij}a^ie^j) \\
    &={} f\dets^{-\frac{1}{2}}\tilde{S}_{\al\be}\Om_{\al\be}e^{123}\no\\ &{}{}\quad+\left(\e_{j\al\be}(f\dets^{-\frac{1}{2}}\tilde{S}_{i\al})_{,\be}-f\dets^{-\frac{1}{2}}\tilde{S}_{i\al}T_{\al j}\right)a^i\hat{e}^j, \no \\
 \label{eq:5_30}   d(f\psi^{\#})&={}  d(f\dets^{-\frac{1}{2}}a^{123}-f\dets^{\frac{1}{2}}a^k\hat{e}^k) \\
   \no &={} \left((f\dets^{\frac{1}{2}})_{,i}+f\dets^{\frac{1}{2}}\e_{i\al\be}T_{\al\be}\right)a^ie^{123} \\
   \no &{}\quad+ f\dets^{-\frac{1}{2}}\Om_{ij}\hat{a}^i\hat{e}^j -(f\dets^{-\frac{1}{2}})_{,i}a^{123}e^i.
\end{align} 
By (\ref{eq:5_27}) and (\ref{eq:5_29}), $\dfrac{\partial \psi}{\partial t}=d(f\om)$ is equivalent to the following:
\begin{align}
\label{eq:5_31}&\pt{\dets}{t}=-f\dets^{-\frac{1}{2}}\tilde{S}_{\al\be}\Om_{\al\be},\quad\pt{e^i}{t}=p_{ik}\hat{e}^k=0, \\ 
\label{eq:5_32}&\de_{ij}q_{\al\al}-q_{ji}=\e_{j\al\be}(f\dets^{-\frac{1}{2}}\tilde{S}_{i\al})_{,\be}-f\dets^{-\frac{1}{2}}\tilde{S}_{i\al}T_{\al j}
\end{align}
for $i,j=1,2,3.$ Assuming 
$T=0$  combined with $S={}^tS$, (\ref{eq:5_32}) is equivalent to 
\begin{align}
 \label{eq:5_33}   q_{ij}&=-\e_{i\al\be}(f\dets^{-\frac{1}{2}}\tilde{S}_{j\al})_{,\be} \\
 \no&= -\e_{i\al\be}f_{,\be}\dets^{-\frac{1}{2}}\tilde{S}_{j\al}-fC_{ij}
\end{align}
for $i,j=1,2,3$, where we use the notation of Lemma \ref{lem:5b}. By (\ref{eq:5_28}) and (\ref{eq:5_30}), $\dfrac{\partial}{\partial t}\left(\dfrac{1}{2}\om\w\om\right)=d(f\psi^{\#})$ is equivalent to the following:
\begin{align}
\label{eq:5_34}    &-\e_{i\al\be}q_{\al\ga}S_{\ga\be}=(f\dets^{\frac{1}{2}})_{,i}+f\dets^{\frac{1}{2}}\e_{i\al\be}T_{\al\be},\\
 \label{eq:5_35}  & S_{i\al}p_{j\al}-S_{ij}p_{\al\al}-\pt{S_{ij}}{t}=f\dets^{-\frac{1}{2}}\Om_{ij},\\
 \label{eq:5_36}  & (f\dets^{-\frac{1}{2}})_{,i}=0
\end{align}
for $i,j=1,2,3.$ Thus, using Lemma \ref{lem:5a} and \ref{lem:5b} and summarizing (\ref{eq:5_26}), (\ref{eq:5_31}), (\ref{eq:5_32}) and (\ref{eq:5_34})--(\ref{eq:5_36}), we obtain the conditions in Proposition \ref{prop:loc_t3_1}.
\end{proof}

Now we can immediately prove Theorem \ref{thm:a} and \ref{thm:a1}.

\begin{proof}[Proof of Theorem \ref{thm:a} and \ref{thm:a1}]
By setting $f=(\det S)^{\frac{1}{2}}$ in Proposition \ref{prop:loc_t3_1}, we obtain Theorem \ref{thm:a}. 
Moreover, by scaling of the parameter $t$, we can deduce Theorem \ref{thm:a1} from Proposition \ref{prop:red_1} and \ref{prop:loc_t3_1}.
\end{proof}

\begin{remark}
We can easily check that the the solutions of (\ref{eq:t1}) preserve the conditions (\ref{eq:t2}) in Theorem \ref{thm:a}. For example, we have 
\begin{align*}
    \pt{S_{i\al,\al}}{t}=\left(\pt{S_{i\al}}{t}\right)_{,\al}=-\Om_{i\al,\al}=0.
\end{align*}
\end{remark}

\begin{remark}\label{rem:ms}
Theorem \ref{thm:a} is essentially equivalent to Theorem 3.5 of \cite{MS}. Their equations of torsion-free conditions for toric $G_2$-manifolds (\cite{MS}, 3.10) is deduced from (\ref{eq:t1}) as follows. This step reminds us of the derivation of the equations of electromagnetic wave from Maxwell's equations. We have 
\begin{align}\label{eq:5_c}
    \frac{\partial \Om_{ij}}{\partial t}= \ep_{i\al\be}\ep_{j\gamma\delta}\frac{\partial^2\tilde{S}_{\al\gamma}}{\partial x^{\be}x^{\delta}}
\end{align}
since 
\begin{align*}
    \pt{da^i}{t}=d\left(\pt{a^i}{t}\right)
\end{align*}
for $i=1,2,3$. Thus, by (\ref{eq:t1}) and (\ref{eq:5_c}), we have 
\begin{align*}
    \frac{\partial^2 S_{ij}}{\partial t^2} =-\ep_{i\al\be}\ep_{j\gamma\delta}\frac{\partial^2\tilde{S}_{\al\ga}}{\partial x^{\be}x^{\delta}}
\end{align*}
for $i,j=1,2,3.$ This is the same as the second-order partial differential equation (\cite{MS}, 3.10). Moreover explicit examples are presented in the paper.
\end{remark}

\section{The case of $\so3$-fibrations}\label{sec:so3}

In this section, we study the case of $\so3$-fibrations. Most calculations are similar to the case of $\rm{T}^3$-fibrations, but in this case we need to use covariant derivation.

Let $G=\so3$, and $P\to M$  a principal $\so3$-bundle over a  3-manifold $M$. Fix a basis $\{Y_1,Y_2,Y_3\}$ of $\fso3$ satisfying $[Y_i,Y_j]=\e_{\al ij}Y_{\al}$ for $i,j=1,2,3.$ 

\subsection{Torsion of $\su3$-structures in the case of $G=\so3$}\label{subsec:6_1}

As have seen in section \ref{sec:slag}, every $\so3$-invariant sLag $\su3$-structure $(\om,\psi)$ on $P$ uniquely decomposes into a triple $(e,a,S)$ of a solder 1-form $e=e^iY_i \in \Om^1(P;\fso3)^{\so3}_{hor}$, a connection 1-form $a=a^iY_i\in \Om^1(P; \fso3)^{\so3}$ and an equivariant $\psym3$-valued function $S=(S_{ij})\in \Om^0(P;\psym3)^{\so3}$.

\begin{remark}
Note that if there exists an $\so3$-invariant sLag $\su3$-structure then $P$ is a trivial $\so3$-bundle over $M$. This is because $P$ is isomorphic to the orthonormal frame bundle over $M$ by the solder 1-form $e$, and any orientable 3-manifolds are parallelizable.
\end{remark}

Let $(\om,\psi)$ be an $\so3$-invariant $\su3$-structure on $P$, and $(e,a,S)$ the triple corresponding to $(\om,\psi)$. Denote by $T=(T_{ij})$ and $\Om=(\Om_{ij})$ the torsion and curvature forms of $(e,a)$: $d_He=de+[a\w e]=T_{ij}\hat{e}^jY_i$ and $d_Ha=da+\dfrac{1}{2}[a\w a]=\Om_{ij}\hat{e}^jY_i$. Also for an equivariant matrix-valued function $A=(A_{ij}) \in \Om^0(P;\mathrm{M}(k;\R))^{\so3}$, denote by $A_{ij;k}$ the covariant derivative: $(d_HA)_{ij}=(dA+a^k[Y_k, A])_{ij}=A_{ij;k}e^k$.

We use the following formulas. The proof is straightforward.

\begin{lemma}\label{lem:6a} Let $A=(A_{ij})\in \Om^0(P;\sym3)^{\so3}$. We have
\begin{align*}
&de^i=T_{i\al}\hat{e}^{\al}-\ep_{i\al\be}a^{\al}e^{\be},\quad da^i=\Om_{i\al}\hat{e}^{\al} - \hat{a}^i ,\\
&d\hat{e}^i=\ep_{i\al\be}T_{\al\be}e^{123}-\ep_{i\al\be}a^{\al}\hat{e}^{\be},\quad
d\hat{a}^i=\ep_{i\al\be}\Om_{\al\gamma}\hat{e}^{\gamma}a^{\be}, \\
&dA_{ij}=A_{ij;\al}e^{\al}+(\ep_{\al i \be}A_{\be j}+\ep_{\al j\be}A_{\be i})a^{\al}
\end{align*}
for $i,j=1,2,3.$
\end{lemma}

By using (\ref{eq:su3c})--(\ref{eq:su3f}) and Lemma \ref{lem:6a}, we obtain

\begin{prop}\label{prop:6b}
Let $(e,a,S)$ be the triple corresponding to $(\om,\psi)$. Then we have
\begin{align}
\label{eq:6_1}&d\omega =\left(\e_{j\al\be}(\dets^{-\frac{1}{2}}\tilde{S}_{i\al})_{;\be} - \dets^{-\frac{1}{2}}\tilde{S}_{i\al}T_{\al j}\right)a^i\hat{e}^j + \dets^{-\frac{1}{2}}\tilde{S}_{ij}\Om_{ij}e^{123}\\
\no&{}\quad\quad +\dets^{-\frac{1}{2}}\tilde{S}_{ij}\hat{a}^ie^j,\\
\no&d\left(\frac{1}{2}\om \w \om\right) = -\left(S_{i\al,\al}+\e_{\al\be\ga}S_{i\al}T_{\be\ga}\right)\hat{a}^ie^{123} ,\\
\no&d\psi = T_{ij}\hat{a}^i\hat{e}^j - \e_{i\al\be}\Om_{\al\be}a^ie^{123},\\
\label{eq:6_4}&d\psi^{\#} = d\left(\dets^{-\frac{1}{2}}\right)a^{123} + \dets^{-\frac{1}{2}}\Om_{ij}\hat{a}^i\hat{e}^j  \\ 
&{}\quad\quad\quad - d\left(\dets^{\frac{1}{2}}\right)a^k\hat{e}^k + \dets^{\frac{1}{2}}\e_{i\al\be}T_{\al\be}a^ie^{123}-\dets^{\frac{1}{2}}\hat{a}^k\hat{e}^k.\no
\end{align}
\end{prop}

\begin{remark}The last terms in (\ref{eq:6_1}) and (\ref{eq:6_4}) are the only terms different from Proposition \ref{prop:4b} in the case of $G=\T3$. The non-commutativity of $\so3$ appears at these points.
\end{remark} 

We have the following corollaries to Proposition \ref{prop:6b}.

\begin{cor}\label{cor:6a} Let $(e,a,S)$ be the triple corresponding to $(\om,\psi)$. 
\begin{enumerate}

\item The $\su3$-structure $(\om,\psi)$ is half-flat if and only if $T_{ij}=0$  and $S_{ik;k}=0$ for $i,j=1,2,3$.

\item If $d\psi=d\psi^{\#}=0$ holds, then the Riemannian metric on $M$ given by the solder $1$-form $e$ is constant negative curvature. 

\item The $\su3$-structure $(\om,\psi)$ is not symplectic.

\end{enumerate}
\end{cor}

\begin{remark}
If $d_He=T_{ij}\hat{e}^jY_i=0$, then the connection $a$ is the Levi-Civita connection of the solder 1-form $e$. Then the Bianchi identity $[d_Ha\w e]=d_Hd_He=0$ holds. Thus $\Om_{ij}=\Om_{ji}$ holds for $i,j=1,2,3.$ Also ${\Om_{ij}}$ coincides with the orthonormal representation of the Einstein-tensor of the Riemannian metric defined by the local coframe $\{e^1,e^2,e^3\}$.
\end{remark}

\begin{remark}
The third statement of Corollary \ref{cor:6a} is immediately derived from the Liouville - Arnold theorem, which implies compact fibers of Lagrangian fibrations of symplectic manifolds are tori.
\end{remark}

\subsection{Local structure of $\so3$-invariant $G_2$-manifolds}

In this subsection, we prove that all torsion-free $\so3$-invariant Lag $G_2$-structures are locally described by orbits of constrained dynamical systems on the spaces of triples $(e,a,S)$. 

As have seen in Section \ref{sec:red}, every {\it coclosed} $\so3$-invariant Lag $G_2$-structure is locally isomorphic to $\om(t)\w f(t)dt+\psi(t)$ on $P\times(t_1,t_2)$, where $(\om(t),\psi(t),f(t))$ is some one-parameter family of $\so3$-invariant sLag $\su3$-structures $(\om,\psi)$ and $\so3$-invariant positive functions $f$ on some $P$, defined on some interval $(t_1,t_2)$. Let $(e(t),a(t),S(t),f(t))$ be a one-parameter family defined on an interval $(t_1,t_2)$ of the triples corresponding to $(\om,\psi)$ and $\so3$-invariant positive functions $f$ on $P$.  We use the notation in Subsection \ref{subsec:6_1}.

\begin{prop}\label{prop:6_3}
The $G_2$-structure $\om(t)\w f(t)dt+\psi(t)$ is torsion-free if and only if the one-parameter family $(e(t),a(t),S(t).f(t))$ satisfies the following equations:
\begin{align}
  \label{eq:6_5}&  T_{ij}=0,\quad S_{i\al;\al}=0,\quad d(f\dets^{-\frac{1}{2}})=0;\\
 \label{eq:6_6} &\pt{e^i}{t}=f\dets^{-\frac{1}{2}}\tilde{S}_{ik}e^k,\\ \label{eq:6_7}&\pt{S_{ij}}{t}=-f\dets^{-\frac{1}{2}}\left(\mathrm{tr}(\tilde{S})S_{ij}-2\dets\de_{ij} + \Om_{ij}\right)
\end{align}
for $i,j=1,2,3$. 
\end{prop}

\begin{remark}
Since $d_He=T_{ij}\hat{e}^jY_i=0$, the behavior of $a(t)$ is determined by (\ref{eq:6_6}). See the lemma below.
\end{remark}

\begin{lemma}\label{lem:6_9}
Suppose that $d_He=0$, $\dfrac{\partial e^i}{\partial t}=p_{ik}e^k$ and $p_{ij}=p_{ji}$ for $i,j=1,2,3$. Then we have
\begin{align*}
    \pt{a^i}{t}=-\e_{i\al\be}p_{j\al;\be}e^j \quad{for}\quad i=1,2,3.
\end{align*}
\end{lemma}

\begin{proof}
We can prove this by direct calculation. See (\cite{Chi}, Lemma 9).
\end{proof}

Note that Lemma \ref{lem:5a} and \ref{lem:5b} hold by replacing the derivatives with covariant derivatives. 

\begin{proof}[Proof of Proposition \ref{prop:6_3}]
Most of the proof is the same as that of Proposition \ref{prop:loc_t3_1}. Set $\dot{e}^i=p_{ij}e^j$ and $\dot{a}^i=q_{ij}e^j$ for $i=1,2,3.$ 

The $G_2$-structure $\om(t)\w f(t)dt+\psi(t)$ is torsion-free if and only if $(\om(t), \psi(t),f(t))$ satisfies the half-flat condition, $\dfrac{\partial \psi}{\partial t}=d(f\om)$ and $\dfrac{\partial }{\partial t}(\dfrac{1}{2}\om\we\om)=d(f\psi^{\#})$ 
 for  each $t$. Let us rewrite these conditions by the one-parameter family $(e(t),a(t),S(t),f(t))$ corresponding to $(\om(t),\psi(t),f(t))$.
 
 By Corollary \ref{cor:6a}, the half-flat condition is equivalent to 
 \begin{align}\label{eq:6hf}
     T_{ij}=0 \quad{and}\quad S_{i\al;\al}=0
 \end{align}
 for $i,j=1,2,3.$
 
 Just as in the proof of Proposition \ref{prop:loc_t3_1}, we have 
 \begin{align}\label{eq:6_9} 
 \pt{\psi}{t} ={} -(\pt{\dets}{t} + \dets p_{\al\al})e^{123} + (\delta_{ij}q_{\al\al}-q_{ji})a^i\hat{e}^j+p_{ij}\hat{a}^ie^j,
 \end{align}
 and
\begin{align}\label{eq:6_10}
\pt{}{t}\left(\frac{1}{2}\om\w\om\right)= -\e_{i\al\be}q_{\al\ga}S_{\ga\be}a^ie^{123}+(S_{i\al}p_{j\al}-S_{ij}p_{\al\al}-\pt{S_{ij}}{t})\hat{a}^i\hat{e}^j.
\end{align}

On the other hand, by Proposition \ref{prop:6b}, we have 
\begin{align}\label{eq:6_11}
    d(f\om) &={} f\dets^{-\frac{1}{2}}\tilde{S}_{\al\be}\Om_{\al\be}e^{123} \\
  \no  &\quad{} + \left(\e_{j\al\be}(f\dets^{-\frac{1}{2}}\tilde{S}_{i\al})_{;\be}-f\dets^{-\frac{1}{2}}\tilde{S}_{i\al}T_{\al j}\right)a^i\hat{e}^j \\
 \no &\quad{}+ f\dets^{-\frac{1}{2}}\tilde{S}_{ij}\hat{a}^ie^j,
\end{align}
and
\begin{align}\label{eq:6_12}
    d(f\psi) &={} \left((f\dets^{\frac{1}{2}})_{;i}+f\dets^{\frac{1}{2}}\e_{i\al\be}T_{\al\be}\right)a^ie^{123} \\
  \no  &\quad{}+ (f\dets^{-\frac{1}{2}}\Om_{ij}-f\dets^{\frac{1}{2}}\delta_{ij})\hat{a}^i\hat{e}^j \\
  \no &\quad {} -(f\dets^{-\frac{1}{2}})_{;i}a^{123}e^i.
\end{align}

By (\ref{eq:6_9}) and (\ref{eq:6_10}), $\dfrac{\partial \psi}{\partial t}=d(f\om)$ is equivalent to the following:
\begin{align}
   \label{eq:6_13} &\pt{\dets}{t} + \dets p_{\al\al}= -f\dets^{-\frac{1}{2}}\tilde{S}_{\al\be}\Om_{\al\be}, \\
   \label{eq:6_14} &\delta_{ij}q_{\al\al}-q_{ji}=\e_{j\al\be}(f\dets^{-\frac{1}{2}}\tilde{S}_{i\al})_{;\be}-f\dets^{-\frac{1}{2}}\tilde{S}_{i\al}T_{\al j}, \\
   \label{eq:6_15} &p_{ij}=f\dets^{-\frac{1}{2}}\tilde{S}_{ij}
\end{align}
for $i,j=1,2,3.$ Assuming $T=0$ combined with $S={}^tS$, in the same way as (\ref{eq:5_33}), (\ref{eq:6_14}) is equivalent to 
\begin{align}\label{eq:6_16}
    q_{ij}&=-\e_{i\al\be}(f\dets^{-\frac{1}{2}}\tilde{S}_{j\al})_{;\be} \\
 \no&= -\e_{i\al\be}f_{;\be}\dets^{-\frac{1}{2}}\tilde{S}_{j\al}-fC_{ij}.
\end{align}

By (\ref{eq:6_10}) and (\ref{eq:6_12}), $\dfrac{\partial }{\partial t}(\dfrac{1}{2}\om\we\om)=d(f\psi^{\#})$ is equivalent to the following:
\begin{align}
  \label{eq:6_17}  &-\e_{i\al\be}q_{\al\ga}S_{\ga\be}=(f\dets^{\frac{1}{2}})_{;i}+f\dets^{\frac{1}{2}}\e_{i\al\be}T_{\al\be},\\
 \label{eq:6_18}   &S_{i\al}p_{j\al}-S_{ij}p_{\al\al}-\pt{S_{ij}}{t}=f\dets^{-\frac{1}{2}}\Om_{ij}-f\dets^{\frac{1}{2}}\delta_{ij}, \\
   \label{eq:6_19} &(f\dets^{-\frac{1}{2}})_{;i}=0
\end{align}
for $i,j=1,2,3.$ By (\ref{eq:6_16}), Lemma \ref{lem:5b} and the assumption of $T=0$, we see that (\ref{eq:6_17}) is equivalent to (\ref{eq:6_19}). Also by Lemma \ref{lem:5a}, (\ref{eq:6_13}) follows from (\ref{eq:6_18}). Thus, using Lemma \ref{lem:6_9} and summarizing (\ref{eq:6hf}), (\ref{eq:6_15}), (\ref{eq:6_16}), (\ref{eq:6_18}) and (\ref{eq:6_19}), we obtain the conditions in Proposition \ref{prop:6_3}.
\end{proof}

Now we can immediately prove Theorem \ref{thm:b} and \ref{thm:b1}.

\begin{proof}[Proof of Theorem \ref{thm:b} and \ref{thm:b1}]
By setting $f=(\det S)^{\frac{1}{2}}$ in Proposition \ref{prop:6_3}, we obtain Theorem \ref{thm:b}. Moreover, by scaling of the parameter $t$, we can deduce Theorem \ref{thm:b1} from Proposition \ref{prop:red_1} and \ref{prop:6_3}.
\end{proof}

\begin{remark}
By Lemma \ref{lem:6_9}, we can see that  if  triples $(e(t), a(t), S(t))$ develope along the equations in Theorem \ref{thm:b}, then we have $\dfrac{\partial a}{\partial t}=-\e_{i\al\be}\tilde{S}_{i\al;\be}e^j$ for $i=1,2,3.$
\end{remark}

\begin{remark}
The equations of motion preserve the constraint conditions in Theorem \ref{thm:b}. This is proved by direct calculation in (\cite{Chi}, Proposition 7). In the paper, we also gave a Hamiltonian formulation of Theorem \ref{thm:b} and some observations on Bryant-Salamon's examples \cite{BS}.
\end{remark}

\begin{example}
Let $M=\su2$, i.e., $P=M\times \so3$. Fix a global section of $P$, and by the pull-back, regard a triple $(e(t),a(t),S(t))$ in Theorem \ref{thm:b} as 1-forms and functions on $M$. Assume $(e(t),a(t),S(t))$ is left-invariant for the group structure on $M$. The equations in Theorem \ref{thm:b} are reduced to the following ordinary differential system. This situation is contained in that of \cite{MS2}. Let $\theta^1,\theta^2,\theta^3$ be left-invariant 1-forms on $M$ satisfying $d\theta^i=-\hat{\theta}^i$ for $i=1,2,3.$ Using $A(t)=(A(t)_{ij}), B(t)=(B(t)_{ij})$ and $C(t)=(C(t)_{ij}) \in \rm{M}(3;\R)$, put  $e^i=A_{ij}e^j$, $a^i=B_{ij}e^j$ and $de^i=C_{ij}\hat{e}^j$ for $i=1,2,3.$ Then $C=-\det(A)\cdot A \cdot {}^tA$, and by simple computation, we can see that (\ref{eq:motion}) in Theorem \ref{thm:b} is equivalent to the following:
\begin{align}
    \frac{dA}{dt} = \tilde{S}A\quad{and} \quad \frac{dS}{dt} = -\mathrm{tr}(\tilde{S})\cdot S +  2\dets\cdot I -C^2 - \tilde{C} + \frac{1}{4}(\mathrm{tr}(C))^2\cdot I .\label{eq:red1}
\end{align}
Here by the  Levi-Civita condition, $B=C-(1/2)\mathrm{tr}(C)\cdot I$.
Moreover, the condition $S_{ij;j}=0$, for $i=1,2,3$, is equivalent to $[S,B]$=0, and this is preserved by solutions of (\ref{eq:red1}). 
\end{example}

\begin{example}
Let $M=\R^3$, i.e., $P=\R^3\times \so3$. Take 1-forms $\theta^1=dr$, $\theta^2=d\rho$ and $\theta^3=\sin(\rho)d\xi$ on $\R^3\setminus \{0\}$, where $(r,\rho, \xi)$ is the polar coordinate on $\R^3$. Using positive functions $f(r,t),g(r,t),k(r,t),l(r,t)$ for $r$ and $t$, put $(e(t)^1,e(t)^2,e(t)^3)=(f(r,t)\theta^1,g(r,t)\theta^2,g(r,t)\theta^3)$ and $S=((k(r,t),0,0),(0,l(r,t),0),(0,0,l(r,t)))$. Then (\ref{eq:const}) and (\ref{eq:motion}) in Theorem \ref{thm:b} are equivalent to the following: 
 \begin{eqnarray*}
 &&\frac{\partial f}{\partial t} = fl^2\quad{and}\quad \frac{\partial k}{\partial t}= kl^2-2k^2l-(\frac{1}{fg}\cdot\frac{\partial g}{\partial r})^2 + \frac{1}{g^2}, \\
 && \frac{\partial g}{\partial t}= gkl\quad{and}\quad \frac{\partial l}{\partial t}= -l^3 - \frac{1}{f^2g}\cdot\frac{\partial^2 g}{\partial r^2} + \frac{1}{f^3g}\cdot\frac{\partial f}{\partial r}\cdot\frac{\partial g}{\partial r}, \\
 &&\frac{\partial k}{\partial r} - \frac{2}{g}\cdot\frac{\partial g}{\partial r}\cdot(l-k)=0.
 \end{eqnarray*}
The last condition is preserved by solutions of the above four equations. Then the trivial solutions corresponding to flat metrics are the following:
 \begin{eqnarray*}
 f(r,t)=\al(2t+\be)^{\frac{1}{2}}, \quad g(r,t)=\al(2t+\be)^{\frac{1}{2}}r,\quad l(r,t)^2=k(r,t)^2=(2t+\be)^{-1},
 \end{eqnarray*}
 where $\al,\be \in \R$.
\end{example}

\subsection*{Acknowledgements}

The author would like to thank N.\ Kawazumi for his encouragement and the anonymous referee for many suggestions. This work was supported by the Leading Graduate Course for Frontiers of Mathematical Sciences and Physics.

\bibliographystyle{alpha}
\bibliography{cplx}

\end{document}